\documentclass[12pt,a4paper]{article}
\usepackage[utf8]{inputenc}
\usepackage[T1]{fontenc}
\usepackage[english]{babel}
\usepackage{amsmath, amssymb, amsfonts, amsthm, mathtools}
\usepackage{bm}
\usepackage{geometry}
\usepackage{hyperref}
\usepackage{color}
\usepackage{graphicx}
\usepackage{tikz}
\usepackage{enumerate}
\usepackage{enumitem}

\allowdisplaybreaks 
\newtheorem{theorem}{Theorem}[section]
\newtheorem{lemma}{Lemma}[section]

\newtheorem{corollary}{Corollary}[section]

\begin{document}

\title{\Large Note on Fractional Sums with Fixed GCD}
\author{Meselem KARRAS}
\date{}
\maketitle

\begin{abstract}
	We investigate fractional sums of arithmetic functions over products of two or three integers, with emphasis on fixed greatest common divisors and multiplicative weights. Let $f$ be an arithmetic function satisfying $f(n) \ll n^\alpha$ for some $0 \le \alpha < 1$. For $r \ge 2$, let $\tau_r(n)$ denote the number of representations of $n$ as a product of $r$ positive integers, and more generally, $\tau_r^{(d)}(n)$ the number of representations with $\gcd$ factors equal to $d$. We establish asymptotic formulas for the fractional sums
	\[
	S_{f,r}^{(d)}(x) = \sum_{n \le x} \tau_r^{(d)}(n) f\!\left(\left\lfloor \frac{x}{n}\right\rfloor \right),
	\]
	in the cases $r=2$ and $r=3$.

\end{abstract}

\section{Introduction}

Let $r \ge 2$ be a fixed integer, and let $f$ be an arithmetic function satisfying
$f(n) \ll n^\alpha$ with $0 \le \alpha < 1$. We consider the fractional sum
\begin{equation}
	S_{f,r}(x)
	=
	\sum_{n \le x}
	\tau_r(n)
	f\!\left(\left\lfloor \frac{x}{n} \right\rfloor\right),
	\label{eq:Sr}
\end{equation}
where $\tau_r(n)$ denotes the number of representations of $n$ as a product of
$r$ positive integers:
\[
\tau_r(n)
=
\#\{(n_1,\dots,n_r)\in\mathbb{N}^r : n_1\cdots n_r = n\}.
\]

For $r=2$, the sum \eqref{eq:Sr} was studied by Karras, Li, and Stucky
\cite{KA-LI-ST}, who proved the asymptotic formula
\begin{equation}
	S_{f,2}(x)
	=
	C_1(f)\, x \log x
	+
	C_2(f)\, x
	+
	O\!\left( x^{\frac{4+3\alpha}{7}+\varepsilon} \right),
	\label{eq:2}
\end{equation}
with constants
\begin{align*}
	C_1(f)
	&=
	\sum_{k \ge 1} \frac{f(k)}{k(k+1)}, \\
	C_3(f)
	&=
	\sum_{k \ge 1}
	f(k)
	\left(
	\frac{\log k}{k}
	-
	\frac{\log(k+1)}{k+1}
	\right), \\
	C_2(f)
	&=
	(2\gamma - 1) C_1(f) - C_3(f),
\end{align*}
where $\gamma$ denotes Euler's constant.

For any fixed integer $d \ge 1$, we define the sum with fixed greatest common divisor
\begin{equation}
	S_{f,r}^{(d)}(x)
	=
	\sum_{n \le x}
	\tau_r^{(d)}(n)
	f\!\left(\left\lfloor \frac{x}{n} \right\rfloor\right),
	\label{eq:Sr_d}
\end{equation}
where $\tau_r^{(d)}(n)$ counts the representations
\[
n = n_1 n_2 \cdots n_r
\quad \text{with} \quad
\gcd(n_1,\dots,n_r)=d.
\]

In particular, for $d=1$, we denote 
\[
\tau_r^{(1)}(n) = \tau_r^\ast(n),
\]
and accordingly
\[
S_{f,r}^{(1)}(x) = S_{f,r}^\ast(x).
\]

A direct connection between these quantities follows from the Möbius inversion formula:
\[
\tau_r^{(d)}(n) = \tau_r^\ast\!\left(\frac{n}{d^r}\right) \quad \text{whenever } d^r \mid n,
\]
and consequently
\[
S_{f,r}^{(d)}(x) = S_{f,r}^\ast\!\left(\frac{x}{d^r}\right).
\]

For real numbers $a$ and $b$ such that $a+b \neq -1$, we also consider the weighted fractional sum
\begin{equation*}
	F_{f}^{a,b,(d)}(x)
	=
	\sum_{\substack{ mn\leq x \\ \gcd(m,n)=d}} m^{b} n^{a} f\!\left( \left\lfloor \frac{x}{mn}\right\rfloor \right) .
\end{equation*}

\medskip

\noindent
In this note, we establish asymptotic formulas for the fractional sums \eqref{eq:Sr} and \eqref{eq:Sr_d}
in the cases $r=2$ and $r=3$, both for unrestricted tuples and under the coprimality condition. Moreover, we provide upper bounds for the weighted sums 
$F_{f}^{a,b,(d)}(x)$. Our results extend the work in \cite{KA-LI-ST} to settings with prescribed greatest common divisors and polynomial weights, thereby generalizing classical divisor-sum estimates to multiplicative interactions under constraints.

\section{Main Result}

\begin{theorem}
	\label{thm:1}
	Let $f$ be an arithmetic function satisfying $f(n)\ll n^{\alpha}$ for some $%
	\alpha \in [0,1)$, and let $d\ge 1$ be an integer such that $d \ll x^{\frac{%
			3-3\alpha}{14}-\varepsilon}$ for a given $\varepsilon>0$. Then 
	\begin{equation*}
		S_{f,2}^{(d)}(x) = \frac{C_1(f)}{\zeta(2)}\,\frac{x}{d^2} \log\!\left(\frac{x}{%
			d^2}\right) + C_f \frac{x}{d^2} + O\!\left(x^{\frac{4+3\alpha}{7}%
			+\varepsilon}\right), \qquad (x\to\infty),
	\end{equation*}
	where 
	\begin{equation*}
		C_f=\sum_{k=1}^{\infty} \frac{\mu(k)}{k^2}\, \left(C_2(f)-2C_1(f)\log
		k\right).
	\end{equation*}
\end{theorem}

\begin{lemma}
	\label{lem:1}
	Let $r\ge 1$ be fixed. For every integer $n\ge 1$, one has
	\[
	\tau_r^{(*)}(n)
	=
	\sum_{\ d^r \mid n}
	\mu(d)\,
	\tau_r\!\left(\frac{n}{d^r}\right),
	\]
	where $\mu$ denotes the Möbius function.
\end{lemma}

\begin{proof}
We use the classical Möbius identity (see e.g., \cite{Apostol}).
	\[
	\mathbf{1}_{\gcd(a_1,\dots,a_r)=1}
	=
	\sum_{d\mid \gcd(a_1,\dots,a_r)} \mu(d).
	\]
	Inserting this identity into the definition of $\tau_r^{(*)}(n)$ yields
	\[
	\tau_r^{*}(n)
	=
	\sum_{\substack{a_1\cdots a_r = n}}
	\sum_{d\mid \gcd(a_1,\dots,a_r)} \mu(d).
	\]
	Since all sums are finite, we may interchange the order of summation to obtain
	\[
	\tau_r^{*}(n)
	=
	\sum_{d\ge 1} \mu(d)
	\sum_{\substack{a_1\cdots a_r = n \\ d\mid a_1,\dots,d\mid a_r}} 1.
	\]
	
	The condition $d\mid a_i$ for all $i=1,\dots,r$ is equivalent to writing
	$a_i = d\,b_i$.
	Then the constraint $a_1\cdots a_r = n$ becomes
	\[
	d^r\, b_1\cdots b_r = n.
	\]
	Hence, such representations exist if and only if $d^r\mid n$, and in this case
	the inner sum equals $\tau_r(n/d^r)$. Therefore,
	\[
\tau_r^{*}(n)
	=
	\sum_{\ d^r \mid n}
	\mu(d)\,
	\tau_r\!\left(\frac{n}{d^r}\right),
	\]
	which completes the proof.
\end{proof}

\begin{lemma}
	\label{lem:2}
	Let $r\ge 2$ and let $f$ be an arithmetic function. For any $x\ge 1$, we have
	\[
	S_{f,r}^{\ast}(x)
	=
	\sum_{d\le x^{1/r}} \mu(d)\,
	S_{f,r}\!\left(\frac{x}{d^r}\right).
	\]

\end{lemma}

\begin{proof}
	By the definition of $S_{f,r}^{\ast}(x)$ and Lemma~\ref{lem:1}, we have 
	\[
	S_{f,r}^{\ast}(x)
	=
	\sum_{n\le x}
	f\!\left(\left\lfloor \frac{x}{n}\right\rfloor\right)
	\sum_{\ d^r \mid n}
	\mu(d)\,
	\tau_r\!\left(\frac{n}{d^r}\right).
	\]
	
	Exchanging the order of summation and setting $n = d^r m$, we obtain
	\[
	S_{f,r}^{\ast}(x)
	=
	\sum_{d\ge 1} \mu(d)
	\sum_{m\le x/d^r}
	\tau_r(m)\,
	f\!\left(\left\lfloor \frac{x/d^r}{m}\right\rfloor\right).
	\]
	
	The inner sum is exactly $S_{f,r}(x/d^r)$, and it is empty for $d>x^{1/r}$.
	Therefore,
	\[
	S_{f,r}^{\ast}(x)
	=
	\sum_{d\le x^{1/r}} \mu(d)\,
	S_{f,r}\!\left(\frac{x}{d^r}\right).
	\]

\end{proof}

\begin{lemma}
\label{lem:3} We have 
\begin{equation*}
\sum_{d>\sqrt{x}} \frac{\mu(d)}{d^2}=O(x^{-1/2}),\qquad \sum_{d>\sqrt{x}} 
\frac{\mu(d)\log d}{d^2}=O\!\left(\frac{\log x}{\sqrt{x}}\right).
\end{equation*}
\end{lemma}

\begin{proof}
Since $|\mu (d)|\leq 1$, 
\begin{equation*}
\sum_{d>\sqrt{x}}\frac{\mu (d)}{d^{2}}\leq \sum_{d>\sqrt{x}}\frac{1}{d^{2}}%
\leq \int_{\sqrt{x}}^{\infty }\frac{dt}{t^{2}}=x^{-1/2}.
\end{equation*}%
For the second estimate, using $|\log d|\leq \log x$ for $d\leq x$, 
\begin{equation*}
\sum_{d>\sqrt{x}}\frac{|\mu (d)\log d|}{d^{2}}\leq (\log x)\sum_{d>\sqrt{x}}%
\frac{1}{d^{2}}=O\!\left( \frac{\log x}{\sqrt{x}}\right) .
\end{equation*}
\end{proof}

\begin{proof}[Proof of Theorem~\ref{thm:1}]
By Lemma \ref{lem:2} and the asymptotic formula \eqref{eq:2}, writing 
\begin{equation*}
\theta =\frac{4+3\alpha }{7},
\end{equation*}%
we have 
\begin{equation*}
S_{f,2}^{(d)}(x)=\sum_{d\leq \sqrt{x}}\mu (d)\left( C_{1}(f)\frac{x}{d^{2}}%
\log \frac{x}{d^{2}}+C_{2}(f)\frac{x}{d^{2}}+O\!\left( \left( \frac{x}{d^{2}}%
\right) ^{\theta +\varepsilon }\right) \right) .
\end{equation*}

After standard computations and using Lemma \ref{lem:3}, we obtain 
\begin{equation*}
S_f^{(1)}(x) = \frac{C_1(f)}{\zeta(2)}x\log x + \left(\frac{C_2(f)}{\zeta(2)}
+ 2 C_1(f)\frac{\zeta^{\prime }(2)}{\zeta^2(2)}\right)x +
O\!\left(x^{\theta+\varepsilon}+x^{1/2}\log x\right).
\end{equation*}

Setting 
\begin{equation*}
M_f=\frac{C_1(f)}{\zeta(2)},\qquad N_f=\frac{C_2(f)}{\zeta(2)} +2C_1(f)\frac{%
\zeta^{\prime }(2)}{\zeta^2(2)},
\end{equation*}
we have 
\begin{equation*}
S_f^{(1)}(x)=M_fx\log x+N_fx +O\!\left(x^{\theta+\varepsilon}+x^{1/2}\log
x\right).
\end{equation*}

Finally, 
\begin{equation*}
S_{f}^{(d)}(x)=S_{f}^{(1)}\!\left( \frac{x}{d^{2}}\right) ,
\end{equation*}%
which completes the proof.
\end{proof}

\begin{theorem}
\label{thm:2} 
We have 
\begin{equation*}
	F_{f}^{a,b,(d)}(x)\ll _{d}d^{b-a-2-2\varepsilon }x^{a+1+\varepsilon }, \qquad (x\to\infty).
\end{equation*}
\end{theorem}

\begin{lemma}
\label{lem:mu-3}For an arithmetic function $f$, such as $f\left( n\right)
\ll n^{\varepsilon }$, $\varepsilon >0,$ and a real $z$ with $z\neq -1,$ 
then 
\begin{equation*}
\sum_{n\leq y}n^{z}f\!\left( \left\lfloor \frac{y}{n}\right\rfloor \right) =%
\frac{c}{z+1}\,y^{z+1}+O\!\left( y^{z+\frac{1}{2}+\varepsilon }\right), \qquad (y\to\infty).
\end{equation*}
\end{lemma}

\begin{proof}

Apply Abel summation to $a(n)=n^{z}$ and 
\begin{equation*}
B(y)=\sum_{n\leq y}f\!\left( \left\lfloor \frac{y}{n}\right\rfloor \right) ,
\end{equation*}%
then 

\[
\sum_{n \le y} n^z f\!\left(\left\lfloor \frac{y}{n} \right\rfloor\right)
= y^z B(y) - \int_1^y z t^{z-1} B(t) \, dt,
\]
By the known formula (see e.g., \cite{Stucky}, \cite{Wu} and \cite{Zhai} ), and in
particular, when $\left\vert f\left( n\right) \right\vert \ll n^{\varepsilon
}$, we have 
\begin{equation*}
B(t)=ct+O(t^{\eta })
\end{equation*}
where 
\begin{equation*}
c:=C_{f}=\sum_{n=1}^{\infty }\frac{f\left( n\right) }{n\left( n+1\right) }%
\text{ and }\eta :=\frac{1}{2}+\varepsilon 
\end{equation*}%
and computing the integral gives the required formula.

\end{proof}

\begin{proof}[Proof of Theorem~\ref{thm:2}]
We start with 
\begin{equation*}
F_{f}^{a,b,(d)}(x)=\sum_{\substack{ mn\leq x \\ \gcd \left( m,n\right) =d}}%
m^{b}n^{a}f\!\left( \left\lfloor \frac{x}{mn}\right\rfloor \right) .
\end{equation*}%
Let $m=du$, $n=dv$ with $\gcd (u,v)=1$, we obtain 
\begin{equation*}
F_{f}^{a,b,(d)}(x)=d^{\,a+b}\sum_{\substack{ uv\leq x/d^{2} \\ \gcd (u,v)=1}}%
u^{b}v^{a}f\!\left( \left\lfloor \frac{x}{d^{2}uv}\right\rfloor \right) .
\end{equation*}%
Put $y=x/d^{2}$ and define 
\begin{equation*}
g(n)=\sum_{\substack{ uv=n \\ \gcd (u,v)=1}}u^{b}v^{a}.
\end{equation*}%
Then the previous equality rewrites as 
\begin{equation*}
	F_{f}^{a,b,(d)}(x)=d^{\,a+b}\sum_{n\leq y}g(n)f\!\left( \left\lfloor \frac{y}{n}%
\right\rfloor \right) .
\end{equation*}

It is easy to check that the function $g(n)$ is multiplicative. Thus, for
any prime $p$ and any integer $\alpha \geq 1$, we have 
\begin{equation*}
g(p^{\alpha })=p^{\alpha a}+p^{\alpha b}\leq 2p^{\alpha a}.
\end{equation*}%
Thus, for all $n$, we have 
\begin{equation*}
g(n)\leq 2^{\omega (n)}n^{a}\ll n^{a+\varepsilon },\text{ }\varepsilon >0
\end{equation*}%
It follows that 
\begin{equation*}
	F_{f}^{a,b,(d)}(x)\ll d^{a+b}\sum_{n\leq y}n^{a+\varepsilon }f\left(
\left\lfloor \frac{y}{n}\right\rfloor \right) ,
\end{equation*}%
and by lemma \ref{lem:mu-3}, we have 
\begin{equation*}
	F_{f}^{a,b,(d)}(x)\ll d^{a+b}\,y^{a+1+\varepsilon },
\end{equation*}%
and replacing $y$ by $x/d^{2},$ we obtain
\begin{equation*}
	F_{f}^{a,b,(d)}(x)\ll _{d}d^{b-a-2-2\varepsilon }\,x^{a+1+\varepsilon }.
\end{equation*}
\end{proof}

\begin{corollary}
Let $d\ge 1$ be an integer and $\varepsilon>0$. Assume $x\to \infty$. 
Consider the arithmetic functions $\tau(n)$ and $\mu^2(n)$, which satisfy  
\begin{equation*}
\tau(n) \ll n^\varepsilon, \quad \mu^2(n) \le 1.  
\end{equation*}
Then, for $d \ll x^{3/14-\varepsilon}$, we have the following formulas:

1) For $\tau (n)$: 
\begin{equation*}
S_{\tau }^{(d)}(x)=\frac{C_{1}(\tau )}{\zeta (2)}\,\frac{x}{d^{2}}\log \frac{%
x}{d^{2}}+C_{\tau }\frac{x}{d^{2}}+O\big(x^{4/7+\varepsilon }\big),
\end{equation*}%
with 
\begin{equation*}
C_{1}(\tau )=\sum_{k\geq 1}\frac{\tau (k)}{k(k+1)},\quad C_{2}(\tau
)=(2\gamma -1)C_{1}(\tau )-\sum_{k\geq 1}\tau (k)\left( \frac{\log k}{k}-%
\frac{\log (k+1)}{k+1}\right) ,
\end{equation*}%
\begin{equation*}
C_{\tau }=\sum_{k\geq 1}\frac{\mu (k)}{k^{2}}\big(C_{2}(\tau )-2C_{1}(\tau
)\log k\big).
\end{equation*}
2) For $\mu^2(n)$:  
\begin{equation*}
S_{\mu^2}^{(d)}(x) = \frac{C_1(\mu^2)}{\zeta(2)}\,\frac{x}{d^2} \log\frac{x}{%
d^2} + C_{\mu^2}\frac{x}{d^2} + O\big(x^{4/7+\varepsilon}\big),  
\end{equation*}
with  
\begin{equation*}
C_1(\mu^2)=\sum_{k\ge 1} \frac{\mu^2(k)}{k(k+1)}, \quad 
C_2(\mu^2)=(2\gamma-1)C_1(\mu^2)-\sum_{k\ge 1} \mu^2(k)\left(\frac{\log k}{k}%
-\frac{\log(k+1)}{k+1}\right),  
\end{equation*}
\begin{equation*}
C_{\mu^2} = \sum_{k\ge 1} \frac{\mu(k)}{k^2} \big(C_2(\mu^2) - 2
C_1(\mu^2)\log k \big).  
\end{equation*}
 
\end{corollary}

The next result provides an upper bound in the case where the arithmetic function $f$ does not satisfy the hypotheses of Theorem~\ref{thm:1}. 
For instance, when $f(n) = \sigma_k(n)$, which grows like $f(n) \ll n^\alpha$ with $\alpha \ge 1$, 
one can obtain a suitable upper bound.

\begin{theorem}
	\label{thm:3} 
	
Let \(k \ge 1\) and \(d \ge 1\) be integers. Then, as \(x \to \infty\), we have
\[
S_{\sigma_k,2}^{(d)}(x) \;\ll\;
\begin{cases}
	\dfrac{x}{d^2} \, \log\!\left(\dfrac{x}{d^2}\right), & \text{if } k = 1,\\[1mm]
	\dfrac{x^k}{d^{2k}}, & \text{if } k > 1.
\end{cases}
\]

	where $$\sigma_k(n) := \sum_{d \mid n} d^k.$$
\end{theorem}

\noindent
We use the following lemma:

\begin{lemma}
	\label{lem:2omega_all_a}
	Let $a\in\mathbb{R}$. As $y\to\infty$, we have
	\[
	\sum_{h\le y}\frac{2^{\omega(h)}}{h^a}
	=
	\begin{cases}
		\displaystyle
		\frac{C_2}{1-a}\,y^{1-a}\log y
		+ O\!\left(y^{1-a}\right),
		& a<1, \\[1.2em]
		
		\displaystyle
		\frac{C_2}{2}(\log y)^2
		+ O(\log y),
		& a=1, \\[1.2em]
		
		\displaystyle
		C_{a}	+ O\!\left(y^{1-a}\log y\right),
		& a>1,
	\end{cases}
	\]
where
\[
C_{a} := \prod_{p} \left( 1 + \frac{2}{p^a} \right), \quad 
C_2 := \prod_{p} \left( 1 - \frac{1}{p} \right)^2 \left( 1 + \frac{2}{p} \right).
\]

\end{lemma}

\begin{proof}
	Let
	\[
	F(t):=\sum_{h\le t}2^{\omega(h)}.
	\]
	It is known (see e.g.\ \cite{O.Bord}) that
	\begin{equation}
		\label{eq:F_asymp}
		F(t)=C_2\,t\log t + O(t),
	\end{equation}
	then, for any real $a$,
	\begin{equation}
		\label{eq:abel}
		\sum_{h\le y}\frac{2^{\omega(h)}}{h^a}
		=
		\frac{F(y)}{y^a}
		+
		a\int_1^y \frac{F(t)}{t^{a+1}}\,dt.
	\end{equation}
	Substituting \eqref{eq:F_asymp} into \eqref{eq:abel} gives
	\[
	\sum_{h\le y}\frac{2^{\omega(h)}}{h^a}
	=
	C_2 y^{1-a}\log y
	+
	aC_2\int_1^y t^{-a}\log t\,dt
	+
	O(y^{1-a}).
	\]
We now distinguish three cases.
	
	\medskip
	\noindent
	\textbf{Case $a<1$.}
	It is immediate that
	\[
	\int_1^y t^{-a}\log t\,dt
	=
	\frac{y^{1-a}}{1-a}\log y
	-
	\frac{y^{1-a}}{(1-a)^2}
	+
	O(1).
	\]
	Hence
	\[
	\sum_{h\le y}\frac{2^{\omega(h)}}{h^a}
	=
	\frac{C_2}{1-a}\,y^{1-a}\log y
	+
	O(y^{1-a}).
	\]

	\medskip
	\noindent
	\textbf{Case $a=1$.}
	We have
	\[
	\sum_{h\le y}\frac{2^{\omega(h)}}{h}
	=
	\frac{C_2}{2}(\log y)^2
	+
	O(\log y).
	\]
	
	\medskip
	\noindent
	\textbf{Case $a>1$.}
	By the Euler product, we have
	\[
	\sum_{h\ge1}\frac{2^{\omega(h)}}{h^a}
	=
	\prod_p\left(1+\frac2{p^a}\right)
	\]
	converges absolutely. Writing
	\[
	\sum_{h\le y}\frac{2^{\omega(h)}}{h^a}
	=
	\sum_{h\ge1}\frac{2^{\omega(h)}}{h^a}
	-
	\sum_{h>y}\frac{2^{\omega(h)}}{h^a},
	\]
and using (5), we have
\[
\sum_{h>y} \frac{2^{\omega(h)}}{h^a} 
\ll \int_y^\infty \frac{F(t)}{t^{a+1}} \, dt  
\ll y^{1-a} \log y,
\]
This completes the proof of the lemma.
\end{proof}

\begin{proof}[Proof of Theorem~\ref{thm:3}]
	Write \(m = du\), \(n = dv\) with \(\gcd(u,v)=1\). Then  
	\[
	S_{\sigma_k}^{(d)}(x) = \sum_{\substack{uv \le x/d^2 \\ \gcd(u,v)=1}} \sigma_k\Big(\Big\lfloor \frac{x}{d^2 uv} \Big\rfloor\Big).
	\]  
	
	Let \(h = uv\). The number of pairs \((u,v)\) with \(\gcd(u,v)=1\) and \(uv=h\) is exactly \(2^{\omega(h)}\), so  
	\[
	S_{\sigma_k}^{(d)}(x) = \sum_{h \le x/d^2} 2^{\omega(h)} \, \sigma_k\Big(\Big\lfloor \frac{x}{d^2 h} \Big\rfloor\Big).
	\]
	
	\medskip
	\textbf{Case \(k>1\):}  
	Use \(\sigma_k(n) \ll n^k\):
	\[
	S_{\sigma_k}^{(d)}(x) \ll \sum_{h \le x/d^2} 2^{\omega(h)} \left(\frac{x}{d^2 h}\right)^k
	= \left(\frac{x}{d^2}\right)^k \sum_{h \le x/d^2} \frac{2^{\omega(h)}}{h^k}.
	\]  
	Since \(\sum_{h\ge 1} 2^{\omega(h)}/h^k\) converges for \(k>1\), we get  
	\[
	S_{\sigma_k}^{(d)}(x) \ll \frac{x^k}{d^{2k}}.
	\]
	
	\medskip
	\textbf{Case \(k=1\):}  
	Use the sharper bound \(\sigma_1(n) \ll n^{1+\varepsilon}\) for any \(\varepsilon>0\), and Lemma~\ref{lem:2omega_sum} with \(a = 1+\varepsilon\):
	\[
	\sum_{h\le y} \frac{2^{\omega(h)}}{h^{1+\varepsilon}} \ll \log y.
	\]
	Hence,
	\[
	S_{\sigma_1}^{(d)}(x) \ll \sum_{h \le x/d^2} 2^{\omega(h)} \left(\frac{x}{d^2 h}\right)^{1+\varepsilon}
	= \frac{x^{1+\varepsilon}}{d^{2(1+\varepsilon)}} \sum_{h \le x/d^2} \frac{2^{\omega(h)}}{h^{1+\varepsilon}} \ll \frac{x}{d^2} \log(x/d^2),
	\]
	for \(\varepsilon>0\) arbitrarily small.
	
\end{proof}

\begin{theorem}
	\label{thm:4}
	Let $f$ be an arithmetic function satisfying $f(n)\ll n^{\alpha}$ for some $0\le \alpha <1$. Then, for any $\varepsilon>0$, we have
	\[
	S_{f,3}^{*}(x) = \frac{C_1(f)}{\zeta(3)} \, x (\log x)^2 + C_2^*(f) \, x \log x + C_3^*(f) \, x + O\Big(x^{1-\delta+\varepsilon}\Big),
	\]
	where $\delta = 43/96$, $\zeta(s)$ is the Riemann zeta function, and
	\[
	C_1(f) := \sum_{d\ge 1} \frac{f(d)}{d},
	\]
	\[
	\begin{aligned}
		C_2^*(f) &:= \sum_{d\ge 1} \frac{\mu(d)}{d^3} \Big( C_2(f) - 6 C_1(f) \log d \Big),\\
		C_3^*(f) &:= \sum_{d\ge 1} \frac{\mu(d)}{d^3} \Big( C_3(f) + 9 C_1(f) (\log d)^2 - 6 C_2(f) \log d \Big),
	\end{aligned}
	\]
	with
	\[
	C_2(f) = (3\gamma-1)\sum_{n\ge 1} \frac{f(n)}{n}, \quad
	C_3(f) = (3\gamma^2 - 3\gamma + 3\gamma_1 + 1)\sum_{n\ge 1} \frac{f(n)}{n},
	\]
	$\gamma$ Euler's constant, and $\gamma_1$ the first Stieltjes constant.
\end{theorem}

The proof of Theorem~\ref{thm:4} relies on the following lemma.
		
\begin{lemma}
	\label{lem:6}
	Let $f$ be an arithmetic function satisfying $f(n)\ll n^{\alpha}$ for some $0\le \alpha <1$. Then, for any $\varepsilon>0$, we have
	\[
	S_{f,3}(x) = C_1(f) x (\log x)^2 + C_2(f) x \log x + C_3(f) x + O\Big(x^{1-\delta+\varepsilon}\Big),
	\]
	where 
	\[
	\begin{aligned}
		C_1(f) &:= \sum_{d\ge 1} \frac{f(d)}{2 d^2},\\
		C_2(f) &:= \sum_{d\ge 1} \frac{f(d)}{d^2} (3\gamma - \log d),\\
		C_3(f) &:= \sum_{d\ge 1} \frac{f(d)}{d^2} \Big( \frac12 (\log d)^2 - 3\gamma \log d + 3\gamma^2 + 3\gamma_1 \Big),
	\end{aligned}
	\]
	with $\gamma$ Euler's constant, $\gamma_1$ the first Stieltjes constant, and $\delta = 43/96$.
\end{lemma}

\begin{proof}
	First, we have
	\[
S_{f,3}(x) = \sum_{n\le x} f\Big(\Big\lfloor \frac{x}{n} \Big\rfloor \Big) \tau_3(n),
	\]
	and split the sum at a parameter $N\in[1,x]$:
	\[
	S_f^{(3)}(x) = E_1 + E_2, \quad
	E_1 := \sum_{n\le N} f\Big(\Big\lfloor \frac{x}{n} \Big\rfloor \Big) \tau_3(n), \quad
	E_2 := \sum_{N<n\le x} f\Big(\Big\lfloor \frac{x}{n} \Big\rfloor \Big) \tau_3(n).
	\]
	
	\medskip
	\noindent
Estimate of $E_1$. Using $\tau_3(n)\ll n^\varepsilon$ and $f(n)\ll n^\alpha$, we get
	\[
	E_1 \ll \sum_{n\le N} (x/n)^\alpha n^\varepsilon \ll x^\alpha N^{1-\alpha+\varepsilon}.
	\]
	
	\medskip
	\noindent
	Estimate of $E_2$. For $n>N$, set $d = \lfloor x/n \rfloor$, so that $x/(d+1)<n\le x/d$. Introducing
	\[
	A_3(t) := \sum_{n\le t} \tau_3(n),
	\]
	we may write
	\[
	E_2 = \sum_{d\le x/N} f(d) \big(A_3(x/d) - A_3(x/(d+1))\big).
	\]
	
	\medskip
	\noindent
	It is known that (see e.g. \cite {Ivić})
	\[
	A_3(t) = t P_2(\log t) + \Delta_3(t), \quad \Delta_3(t) \ll t^{\delta+\varepsilon},
	\]
	with 
	\[
	P_2(u) = \frac12 u^2 + (3\gamma-1)u + (3\gamma^2-3\gamma+3\gamma_1+1).
	\]
	
	\noindent
	Using a Taylor expansion around $u = \log(x/d)$ with $h = -1/d + O(1/d^2)$, we obtain
	\[
	P_2(\log(x/(d+1))) = P_2(\log(x/d)) - \frac{P_2'(\log(x/d))}{d} + O\Big(\frac{\log x}{d^2}\Big),
	\]
	and therefore
	\[
	\frac{P_2(\log(x/d))}{d} - \frac{P_2(\log(x/(d+1)))}{d+1} 
	= \frac{P_2(\log(x/d)) + P_2'(\log(x/d))}{d^2} + O\Big(\frac{\log x}{d^3}\Big).
	\]
	
	\noindent
Writing 
\[
P_2(\log(x/d)) + P_2'(\log(x/d)) = \frac{1}{2} (\log(x/d))^2 + 3 \gamma \, \log(x/d) + \tilde{c}, 
\quad \text{with} \quad \tilde{c} = 3\gamma^2 + 3\gamma_1,
\]
then, we obtain

	\[
	\sum_{d\le x/N} f(d) \big(A_3(x/d) - A_3(x/(d+1))\big)
	= x C_1(f) (\log x)^2 + x C_2(f) \log x + x C_3(f) + O\Big(\frac{x}{N}\Big),
	\]
	with the constants $C_i(f)$ as defined above.
	
	\medskip
	\noindent
   From the error term, we have
   \[
   \sum_{d \le x/N} |f(d)| \, \Delta_3(x/d) \ll x^{1+\alpha+\varepsilon} \, N^{-(1+\alpha-\delta)}.
   \]

	\medskip
	\noindent
Finally, by balancing the error terms, we choose $N = x^{1/(2-\delta)}$, which results in a total error of
\[
O\big(x^{1-\delta+\varepsilon}\big).
\]

\medskip
\noindent
Collecting all contributions, we then obtain
\[
S_f^{(3)}(x) = C_1(f) \, x (\log x)^2 + C_2(f) \, x \log x + C_3(f) \, x + O\big(x^{1-\delta+\varepsilon}\big).
\]

\end{proof}

\begin{proof}[Proof of Theorem~\ref{thm:4}]
First, by Lemma~\ref{lem:2}, we have
\[
S_{f,3}^{\ast}(x)
=
\sum_{d \le x^{1/3}} \mu(d)\,
S_{f,3}\!\left(\frac{x}{d^3}\right).
\]

Applying Lemma~\ref{lem:6} with $y = x/d^3$ then gives
\[
S_{f,3}^{\ast}(x) = \sum_{d \ge 1} \mu(d) \Bigg[ 
C_1(f) \frac{x}{d^3} (\log(x/d^3))^2 
+ C_2(f) \frac{x}{d^3} \log(x/d^3) 
+ C_3(f) \frac{x}{d^3} 
+ O\Big((x/d^3)^{1-\delta+\varepsilon}\Big) 
\Bigg],
\]

and we have
\[
\log(x/d^3) = \log x - 3 \log d, \quad (\log(x/d^3))^2 = (\log x)^2 - 6 (\log x)(\log d) + 9 (\log d)^2.
\]

	Hence,
	\begin{align*}
		\sum_{d\ge 1} \frac{\mu(d)}{d^3} C_1(f) x (\log(x/d^3))^2
		&= x (\log x)^2 C_1(f) \sum_{d\ge 1} \frac{\mu(d)}{d^3} 
		- 6 C_1(f) x (\log x) \sum_{d\ge 1} \frac{\mu(d) \log d}{d^3} \\
		&\quad + 9 C_1(f) x \sum_{d\ge 1} \frac{\mu(d) (\log d)^2}{d^3}.
	\end{align*}
	
	Since $\sum_{d\ge 1} \mu(d)/d^3 = 1/\zeta(3)$, we obtain the main term
	\[
	\frac{C_1(f)}{\zeta(3)} x (\log x)^2,
	\]
	and the remaining terms contribute to $C_2^*(f)$ and $C_3^*(f)$ as defined above.
	
	Finally, the error term satisfies
	\[
	\sum_{d\ge 1} O\Big((x/d^3)^{1-\delta+\varepsilon}\Big) \ll O(x^{1-\delta+\varepsilon}).
	\]
 This completes the proof.
\end{proof}

\bigskip \noindent \newline
Meselem KARRAS,\newline
Faculty of Science and Technology, Tissemsilt University, Algeria. \newline
FIMA Laboratory, Khemis Miliana University, Algeria. \newline
Email: \texttt{m.karras@univ-tissemsilt.dz}

\end{document}